\documentclass[11pt]{amsart}

\usepackage{latexsym}
\usepackage{float}
\usepackage{amssymb}
\usepackage{amsmath}
\usepackage{mathrsfs}
\usepackage{euscript}
\usepackage[all]{xy}
\newcommand{\gb}{\beta}
\newcommand{\ga}{\alpha}

\newcommand{\gd}{\delta}
\newcommand{\gD}{\Delta}
\renewcommand{\ge}{\epsilon}

\newcommand{\fg}{{\mathfrak g}}
\newcommand{\fh}{{\mathfrak h}}

\newcommand{\fl}{{\mathfrak l}}

\newcommand{\fn}{{\mathfrak n}}

\newcommand{\fq}{{\mathfrak q}}

\newcommand{\fz}{{\mathfrak z}}
\newcommand{\f}{\mathfrak}

\newcommand{\eC}{\EuScript{C}}
\newcommand{\eD}{\EuScript{D}}

\newcommand{\eM}{\EuScript{M}}

\newcommand{\eX}{\EuScript{X}}

\newcommand{\nbar}{\bar{n}}                 
\newcommand{\Nbar}{\bar{N}}

\newcommand{\C}{\mathbb{C}}          

\newcommand{\Z}{\text{\bf Z}}

\newtheorem{Thm}[equation]{Theorem}
\newtheorem{Lem}[equation]{Lemma}
\newtheorem{Cor}[equation]{Corollary}
\newtheorem{Prop}[equation]{Proposition}

\newtheorem{Def}[equation]{Definition}

\numberwithin{equation}{section}
\newcommand{\be}{\begin{equation}}
\newcommand{\beu}{\begin{equation*}}

\newcommand{\acts}{ {\raisebox{1pt} {$\scriptstyle \bullet$} } }
\newcommand{\ad}{\text{ad}}
\newcommand{\Ad}{\text{Ad}}
\newcommand{\Cal}{\mathcal}

\newcommand{\IP}[2]{\langle#1 , #2\rangle}     

\newcommand{\tOmega}{\tilde{\Omega}}

\begin{document}

\bibliographystyle{amsplain}

\baselineskip=16pt

\title[Conformally Invariant Systems]
{A System of Third-Order Differential Operators Conformally Invariant under $\f{so}(8,\mathbb{C})$}

\author{Toshihisa Kubo}
\address{Department of Mathematics,
             Oklahoma State University,
              Stillwater, Oklahoma 74078}
\email{ toskubo@math.okstate.edu}

\subjclass[2010]{Primary 22E46; Secondary 22E47, 17B10}
\keywords{differential intertwining operators, generalized Verma modules, real flag manifolds}

\begin{abstract}
In earlier work, Barchini, Kable, and Zierau constructed a number of conformally
invariant systems of differential operators associated to Heisenberg parabolic
subalgebras in simple Lie algebras. The construction was systematic,
but the existence of such a system was left open in several anomalous cases.
Here, a conformally invariant system is shown to exist in the most interesting
of these remaining cases. The construction may also be interpreted as 
giving an explicit homomorphism between generalized Verma modules
for the Lie algebra of type $D_4$.
\end{abstract}
\maketitle

\section{Introduction}

Conformally invariant systems of differential operators 
on a manifold $M$ on which a Lie algebra $\fg$ acts by first order differential operators
were studied by Barchini, Kable, and Zierau in \cite{BKZ08} and \cite{BKZ09}.   
Loosely speaking, a conformally invariant system is a list of differential operators 
$D_1, \ldots, D_m$ that satisfies the bracket identity
\beu
[\Pi(X), D_j] = \sum_{i}C_{ij}(X)D_i,
\end{equation*}
where $\Pi(X)$ is the differential operator corresponding to $X \in \fg$
and $C_{ij}(X)$ are smooth functions on $M$.
We shall give the definition of conformally invariant systems more precisely
in Section 2. 
While a general theory of conformally invariant systems is developed in \cite{BKZ09},  
examples of such systems of differential operators 
associated to the Heisenberg parabolic subalgebras
of any complex simple Lie algebras are constructed in \cite{BKZ08}.
The purpose of this paper is to answer a question, left open in 
\cite{BKZ08}, concerning the existence of a certain conformally invariant
system of third-order differential operators. 
This is done by constructing the required system.
This result may be interpreted as giving an explicit homomorphism between
two generalized Verma modules, one of which is non-scalar. 
The problem of constructing and classifying homomorphisms between 
scalar generalized Verma modules has received a lot of attention;
for recent work, see, for example, \cite{Matumoto06}.
Much less is known about maps between generalized Verma modules 
that are not necessarily scalar.


In order to explain our main results in this paper,
we briefly review the results of \cite{BKZ08} here. 
To begin with, 
let $\fg$ be a complex simple Lie algebra and $\fq = \fl \oplus \fn$
be the parabolic subalgebra of Heisenberg type; that is,
$\fn$ is a two-step nilpotent algebra with one-dimensional center. 
We denote by $\gamma$ the highest root of $\fg$.
For each root $\ga$ let $\{X_{-\ga}, H_{\ga}, X_\ga\}$
be a corresponding  $\f{sl}(2)$-triple,
normalized as in Section 2 of \cite{BKZ08}.
Then $\ad(H_\gamma)$ on $\fg$ has eigenvalues $-2$, $-1$, $0$, $1$, $2$, 
and the corresponding eigenspace decomposition of $\fg$ is denoted by
\beu
\fg = \fz(\bar\fn) \oplus V^- \oplus \fl \oplus V^+ \oplus \fz(\fn).
\end{equation*}

Let $\mathbb{D}[\fn]$ be the Weyl algebra of $\fn$.
Then each system constructed in \cite{BKZ08} derives from a $\C$-linear map 
$\Omega_k: \fg(2-k) \to \mathbb{D}[\fn]$ with $1\leq k \leq 4$
and $\fg(2-k)$ the $2-k$ eigenspace of $\ad(H_\gamma)$.
Let $\Pi_s: \fg \to \mathbb{D}[\fn]$
be the Lie algebra homomorphism constructed in Section 4 in \cite{BKZ08}.
Here $s$ is a complex parameter.
We say that the $\Omega_k$ system has special value $s_0$
when the system is conformally invariant for $\Pi_{s_0}$.

In \cite{BKZ08} the special values of $s$ are determined
for the $\Omega_k$ systems with $k=1,2,4$
for all complex simple Lie algebras, but 
only exceptional cases are considered for the $\Omega_3$ system.
A table in Section 8.10 in \cite{BKZ08} lists the special values of 
$s$. The reader may want to notice that
the entries in the columns for the systems 
$\Omega_2^{\text{big}}$ and $\Omega_2^{\text{small}}$ 
for types $B_r$ and $C_r$ should be transposed.
Theorem 21 in \cite{BKZ09} then shows that the $\Omega_3$ system
does not exist for $A_r$ with $r \geq 3$, $B_r$ with $r\geq 3$, and 
$D_r$ with $r \geq 5$.
There remain two open cases, namely, the $\Omega_3$ system
for type $A_2$ and the $\Omega_3$ system for type $D_4$. 
The aim of this paper is to show 
that the $\Omega_3$ system does exist for type $D_4$
(see Theorem \ref{Thm311}).
In order to achieve the result we use several facts 
from both \cite{BKZ08} and \cite{BKZ09}.
By using these facts, we significantly 
reduce the amount of computation to show the existence of the system.
In the other remaining case, for the algebra of type $A_2$, 
the Heisenberg parabolic subalgebra coincides with the Borel subalgebra,
and the existence of the $\Omega_3$ system(s) follows from 
the standard reducibility result for Verma modules 
(see for instance \cite[Theorem 7.6.23]{Dix96}).


There are two differences between our conventions here and those used in \cite{BKZ08}.
One is that we choose the parabolic $Q_0 = L_0 N_0$ for the real flag manifold,
while the opposite parabolic $\bar Q_0 = L_0 \bar N_0$ is chosen in \cite{BKZ08}.
Because of this, our special values of $s$
are of the form $s = -s_0$, where $s_0$ 
are the special values shown in Section 8.10 in \cite{BKZ08}.
The other is that we identify $(V^+)^*$ with $V^-$ by using the Killing form,
while $(V^+)^*$  in \cite{BKZ08} is identified with $V^+$
by using the non-degenerate alternating form $\IP{\cdot}{\cdot}$ on $V^+$
defined by $[X_1, X_2] = \IP{X_1}{X_2}X_\gamma$ for $X_1, X_2 \in V^+$. 
Because of this difference the right action $R$, which will be defined
in Section 2, will play the role played by $\Omega_1$ in \cite{BKZ08}.
In addition to these notational differences, there are also some methodological 
differences between \cite{BKZ08} and what we do here. 
These stem from the fact that we make systematic use of the results
of \cite{BKZ09} to streamline the process of proving conformal invariance. 

We now outline the remainder of this paper.
In Section 2, we review the setting and results of Section 5 in \cite{BKZ09},
simultaneously specializing them to the situation considered here.
It would be helpful for the reader to be familiar with \cite{BKZ09},
particularly the concepts of $\fg$-manifold and $\fg$-bundle, 
at this point; the definitions may be found on pp. 790-791 of \cite{BKZ09}.
In Section 3, we specialize further by taking $\fg$ to be of type $D_4$.
We fix a suitable Chevalley basis and give the definition of the 
$\tilde{\Omega}_3$ system whose conformal invariance is to be established.
A remark on notation might be helpful here. In \cite{BKZ08},
a system $\Omega'_3$ is initially defined. It is then shown to decompose
as a sum of a leading term $\tOmega_3$ and a correction term $C_3$. 
These two are recombined with different coefficients to give $\Omega_3$,
which is finally shown to be conformally invariant for exceptional algebras.
For type $D_4$, it emerges that $\Omega_3 = \tOmega_3$,
so that the correction term $C_3$ is discarded completely. 
For this reason, we do not recapitulate the process. Rather,
we simply introduce $\tOmega_3$ and proceed to show that 
it is conformally invariant. This is done in Theorem 3.11, which is our main result.

\vskip 0.1in

\noindent \textbf{Acknowledgment.} The author would like to thank 
Dr. Anthony Kable for his valuable suggestions and comments on this paper.

\section{Conformally Invariant Systems}

The purpose of this section is to introduce the notion of conformally invariant systems.
Let $G_0$ be a connected real semisimple Lie group with Lie algebra $\fg_0$ and 
complexified Lie algebra $\fg$. Let $Q_0$ be a parabolic subgroup of $G_0$ and 
$Q_0 = L_0N_0$ a Levi decomposition of $Q_0$. By the Bruhat decomposition,
the subset $\bar{N_0}Q_0$ of $G_0$ is open and dense in $G_0$,
where $\bar{N_0}$ is the nilpotent subgroup of $G_0$ opposite to $N_0$.
Let $\bar\fn$ and $\fq$ be the complexifications of 
the Lie algebras of $\bar N_0$ and $Q_0$, respectively;
we have the direct sum $\fg = \bar \fn \oplus \fq$.
For $Y \in \fg$, write $Y = Y_{\bar \fn} + Y_{\fq}$ for the decomposition 
of $Y$ in this direct sum.
Similarly, write the Bruhat decomposition of $g \in \bar N_0 Q_0$ as 
$g= \mathbf{\bar n}(g)\mathbf{q}(g)$ with $\mathbf{\bar n}(g) \in \bar N_0$ and 
$\mathbf{q}(g) \in Q_0$. Note that for $Y \in \fg_0$ we have
\beu
Y_{\bar \fn} = \frac{d}{dt} \mathbf{\bar n}(\exp(tY)) \big|_{t=0},
\end{equation*}
and a similar equality holds for $Y_{\fq}$.

We consider the homogeneous space $G_0/Q_0$. Let $\C_{\chi^{-s}}$ be 
the one-dimensional representation of $L_0$ with character $\chi^{-s}$. 
The representation $\chi^{-s}$ is extended to a representation of $Q_0$
by making it trivial on $N_0$.
For any manifold $M$, denote by $C^\infty(M,\C_{\chi^{-s}})$ the smooth
functions from $M$ to $\C_{\chi^{-s}}$. The group $G_0$ acts on the space
\beu
C^\infty_{\chi}(G_0/Q_0, \C_{\chi^{-s}}) = \{ F \in C^\infty(G_0, \C_{\chi^{-s}}) \; |\;
\text{$F(gq) = \chi^{-s}(q^{-1})F(g)$ for all $q \in Q_0$ and $g \in G_0$} \}
\end{equation*}
\vskip 0.1in
\noindent by left translation,
and the action $\Pi_s$ of $\fg$ on $C^\infty_\chi(G_0/Q_0, \C_{\chi^{-s}})$
arising from this action is given by
\beu
(\Pi_s(Y) \acts F)(g) = \frac{d}{dt}F(\exp(-tY)g)\big|_{t=0}
\end{equation*}
\vskip 0.1in

\noindent for $Y \in \fg_0$. Here the dot $\bullet$ denotes the action of 
$\Pi_s(Y)$. This action is extended $\C$-linearly
to $\fg$ and then naturally to the universal enveloping algebra $\Cal{U}(\fg)$.
We use the same symbols for the extended actions.

The restriction map 
$C^\infty_{\chi}(G_0/Q_0, \C_{\chi^{-s}}) \to C^\infty(\bar N_0, \C_{\chi^{-s}})$
is an injection whose image is dense for the smooth topology.
We may define the action 
of $\Cal{U}(\fg)$ on the image of the restriction map by 
$\Pi_s(u) \acts f = \big(\Pi_s(u) \acts F\big)|_{\bar N_0}$ for $u \in \Cal{U}(\fg)$ 
and $F \in C^\infty_{\chi}(G_0/Q_0, \C_{\chi^{-s}})$ with $f = F|_{\bar N_0}$. 
Define a right action
$R$ of $\Cal{U}(\bar \fn)$ on $C^\infty(\bar N_0, \C_{\chi^{-s}})$ by
\beu
\big( R(X) \acts f \big)(\nbar) = \frac{d}{dt} f\big( \nbar \exp(tX) \big) \big|_{t=0}
\end{equation*}
\vskip 0.1in

\noindent for $X \in \bar\fn_0$ and $f \in C^\infty(\bar N_0 ,\C_{\chi^{-s}})$. 
A direct computation shows that
\begin{equation}\label{Eqn21}
\big(\Pi_s(Y) \acts f \big)(\nbar) = -sd\chi \big( (\Ad(\nbar^{-1})Y)_\fq \big) f(\nbar) 
-\big( R \big( (\Ad(\nbar^{-1})Y)_{\bar \fn} \big) \acts f \big)(\nbar)
\end{equation}
\vskip 0.1in

\noindent for $Y \in \fg$ and $f$ in the image of the restriction map 
$C^\infty_{\chi}(G_0/Q_0, \C_{\chi^{-s}}) \to C^\infty(\bar N_0, \C_{\chi^{-s}})$. 
This equation implies that the representation $\Pi_s$ extends 
to a representation of $\Cal{U}(\fg)$ on 
the whole space $C^\infty(\bar{N_0} , \C_{\chi^{-s}})$. 
Note that for all $Y \in \fg$,
the linear map $\Pi_s(Y)$ is in $C^\infty(\bar{N_0}, \C_{\chi^{-s}}) \oplus \eX(\bar{N_0})$,
where $\eX(\bar{N_0})$ is the space of smooth vector fields on $\bar{N_0}$.
This property of $\Pi_s(Y)$
makes $\bar{N_0}$ a \emph{$\fg$-manifold} in the sense of \cite[page 790]{BKZ09}. 

Let $\Cal{L}_{-s}$ be the trivial bundle of $\bar{N_0}$ with fiber $\C_{\chi^{-s}}$.
Then the space of smooth sections of $\Cal{L}_{-s}$
is identified with $C^\infty(\bar{N_0} , \C_{\chi^{-s}})$. An operator
$D : C^\infty(\bar{N_0} , \C_{\chi^{-s}}) \to C^\infty(\bar{N_0} , \C_{\chi^{-s}})$ 
is said to be a 
\emph{differential operator} if it is of the form 

\beu
D = \sum_{|\ga| \leq k} a_\ga \frac{\partial^\ga}{\partial x^\ga},
\end{equation*}
\vskip 0.1in

\noindent where $a_\ga \in C^\infty(\bar{N_0},\C_{\chi^{-s}})$, $k \in \Z_{\geq 0}$,
and multi-index notation is being used.
 
Denote the space of differential operators by $\mathbb{D}(\Cal{L}_{-s})$.
The elements of $C^\infty(\bar{N_0},\C_{\chi^{-s}})$ may be regarded as 
differential operators by identifying them with 
the multiplication operator they induce.
A computation shows that
in $\mathbb{D}(\Cal{L}_{-s})$,
\beu
\big( [\Pi_s(Y),f] \big)(\nbar) =-\big( R \big( (\Ad(\nbar^{-1})Y)_{\bar \fn} \big) \acts f \big)(\nbar) 
\end{equation*}
\vskip 0.1in

\noindent for $Y \in \fg$ and $f \in C^\infty(\bar{N_0})$. 
This verifies that $\Pi_s$ gives $\Cal{L}_{-s}$ the structure of a 
\emph{$\fg$-bundle} in the sense of \cite[page 791]{BKZ09}.

\begin{Def}
Let $\Pi_s$ and $\Cal{L}_{-s}$ be as above.
A \emph{conformally invariant system}
on $\Cal{L}_{-s}$ with respect to $\Pi_s$ is a list of differential operators
$D_1, \ldots, D_m \in \mathbb{D}(\Cal{L}_{-s})$ so that 
the following two conditions are satisfied:
\vskip 0.1in

\begin{enumerate}
\item[(C1)] The list $ D_1, \ldots, D_m $ is linearly independent at each point of $\bar{N_0}$. 
\item[(C2)] For each $Y \in \fg$ there is an $m \times m$ matrix $C(Y)$ of smooth functions
on $\bar{N_0}$ so that, in $\mathbb{D}(\Cal{L}_{-s})$,
\beu
[\Pi_s(Y), D_j] = \sum_i C_{ij}(Y)D_i.
\end{equation*}
\end{enumerate}
\vskip 0.1in
\end{Def}
\noindent The map $C : \fg \to M_{m \times m}(C^\infty(\bar{N_0}))$ is  called the 
\emph{structure operator}.

Now we define 
\beu
\mathbb{D}(\Cal{L}_{-s})^{\bar \fn}
=\{ D \in \mathbb{D}(\Cal{L}_{-s}) \; | \; [\Pi_s(X), D] = 0 
\text{ for all $X \in \bar \fn$} \}.
\end{equation*}
\vskip 0.1in

\begin{Prop}\cite[Proposition 13]{BKZ09}\label{Prop22}
Let $D_1, \ldots, D_m$ be a list of operators in $\mathbb{D}(\Cal{L}_{-s})^{\bar \fn}$.
Suppose that the list is linearly independent at $e$ and that there is a map
$b: \fg \to \f{gl}(m,\C)$ such that

\beu
\big([\Pi_s(Y), D_i] \acts f\big)(e) = \sum_{j=1}^m b(Y)_{ji}(D_j \acts f)(e)
\end{equation*} 
\vskip 0.1in
\noindent for all $Y \in \fg, \; f \in C^\infty(\bar{N_0}, \C_{\chi^{-s}})$, and $1\leq i \leq m$. 
Then $D_1, \ldots, D_m$ is a conformally invariant system on $\Cal{L}_{-s}$.
The structure operator of the system is given by 
$C(Y)(\nbar) = b(\emph{\Ad}(\nbar^{-1})Y)$ for all $\nbar \in \bar{N_0}$ and $Y \in \fg$.
\end{Prop}
As shown on p.802 in \cite{BKZ09} the differential operators in 
$\mathbb{D}(\Cal{L}_{-s})^{\bar \fn}$ can be described in terms of 
elements of the generalized Verma module
\beu
\eM(\C_{sd\chi}) = \Cal{U}(\fg) \otimes_{\Cal{U}(\fq)}\C_{sd\chi},
\end{equation*}
where $\C_{sd\chi}$ is the $\fq$-module derived from 
the $Q_0$-representation $(\chi^{s}, \C)$. 
By identifying $\eM(\C_{sd\chi}) $ as $\Cal{U}(\bar\fn) \otimes \C_{sd\chi}$, 
the map $\eM(\C_{sd\chi})  \to \Cal{U}(\bar \fn)$ given by 
$u \otimes 1 \mapsto u$ is an isomorphism. 
The composition
\begin{equation}\label{Eqn23}
\eM(\C_{sd\chi})  \to \Cal{U}(\bar \fn) \to \mathbb{D}(\Cal{L}_{-s})^{\bar \fn}
\end{equation}
is then a vector-space isomorphism, where 
the map $\Cal{U}(\bar \fn) \to \mathbb{D}(\Cal{L}_{-s})^{\bar \fn}$ is given by
$u \mapsto R(u)$.

Suppose that $f \in C^\infty(\bar{N_0},\C_{\chi^{-s}})$ and $l \in L_0$. Then we define 
an action of $L_0$ on $C^\infty(\bar{N_0},\C_{\chi^{-s}})$ by
\beu
(l \cdot f)(\nbar) = \chi^{-s}(l)f(l^{-1}\nbar l).
\end{equation*} 
This action agrees with the action of $L_0$ by left translation on the image 
of the restriction map $C^\infty_{\chi}(G_0/Q_0, \C_{\chi^{-s}}) \to 
C^\infty(\bar N_0, \C_{\chi^{-s}})$.
In terms of this action we define an action of $L_0$ on $\mathbb{D}(\Cal{L}_{-s})$ by
\beu
(l \cdot D) \acts f = l \cdot \big(D \acts (l^{-1} \cdot f)\big).
\end{equation*}
\noindent One can check that we have $l \cdot R(u) = R(\Ad(l)u)$ for $l \in L_0$ and 
$u \in \Cal{U}(\bar \fn)$; in particular this $L_0$-action stabilizes the subspace
$\mathbb{D}(\Cal{L}_{-s})^{\bar \fn}$. Also $L_0$ acts on $\eM(\C_{sd\chi}) $ by 
$l \cdot (u \otimes z) = \Ad(l)u \otimes z$, and with these actions,
the isomorphism (\ref{Eqn23}) is $L_0$-equivariant.
For $D\in \mathbb{D}(\Cal{L}_{-s})$, we denote by $D_{\nbar}$ the linear functional
$f \mapsto (D\acts f)(\nbar)$ for $f \in C^\infty(\Nbar_0, \C_{\chi^{-s}})$. 
The following result is the specialization of Theorem 15 in \cite{BKZ09}
to the present situation.

\begin{Thm}\label{Thm24}
Suppose that $F$ is a finite-dimensional $\fq$-submodule of 
the generalized Verma module $\eM(\C_{sd\chi}) $. Let $f_1, \ldots, f_k$ be a basis of $F$
and define constants $a_{ri}(Y)$ by
\beu
Yf_i = \sum_{r=1}^k a_{ri}(Y)f_r
\end{equation*}
for $1\leq i \leq k$ and $Y \in \fq$. Let $D_1, \ldots, D_k \in 
\mathbb{D}(\Cal{L}_{-s})^{\bar \fn}$ correspond to the elements 
$f_1, \ldots, f_k \in F$. Then
\beu
[\Pi_s(Y), D_i]_{\nbar} 
= \sum_{r=1}^k a_{ri}\big( (\emph{Ad}(\nbar ^{-1})Y)_\fq\big) (D_r)_{\nbar}
-sd\chi\big( (\emph{Ad}(\nbar^{-1})Y)_\fq\big) (D_i)_{\nbar}
\end{equation*}
for all $Y \in \fg$, $1\leq i \leq k$, and $\nbar \in \Nbar_0$.
\end{Thm}


\section{The $\Omega_3$ System on $\f{so}(8,\C)$}

In this section, we specialize to the situation where $G_0$ is a real form 
of the group $SO(8,\C)$ that contains a real parabolic subgroup of Heisenberg type.
In this setting, we construct a system of differential operators on the bundle
$\Cal{L}_1$ and show that it is conformally invariant.
We first introduce some notation.

Let $\fg = \f{so}(8,\C)$.
Choose a Cartan subalgebra $\fh$ of $\fg$
and let $\gD$ be the set of roots of $\fg$ with respect to $\fh$. 
Fix $\gD^+$ a positive system and denote by $S$ 
the corresponding set of simple roots.
We denote the highest root by $\gamma$.
Let $B_\fg$ denote a positive multiple of the Killing form on $\fg$ and denote
by $(\cdot, \cdot)$ the corresponding inner product induced on $\fh^*$.
The normalization of $B_\fg$ will be specified below.
Let us write $||\ga||^2 = (\ga, \ga)$ for any $\ga \in \gD$.
For $\ga \in \gD$, we let $\fg_{\ga}$ 
be the root space of $\fg$ corresponding to $\ga$.
For any $\ad(\fh)$-invariant subspace $V \subset \fg$,
we denote by $\gD(V)$ the set of roots $\ga$ so that $\fg_\ga \subset V$.

It is known that we can choose $X_\ga \in \fg_\ga$ and $H_\ga \in \fh$ for each $\ga \in \gD$
in such a way that the following conditions hold. 
The reader may want to note that our normalizations are special cases of
those used in \cite{BKZ08}.
\vskip 0.1in

\begin{enumerate}
\item[(C1)] For each $\ga \in \gD^+$,  $\{X_{-\ga}, H_{\ga},X_\ga \}$ is an $\f{sl}(2)$-triple.
In particular, 
\beu
[X_\ga, X_{-\ga}] = H_\ga.
\end{equation*} \vskip 0.05in
\item[(C2)] For each $\ga, \gb \in \gD$, $[H_\ga, X_\gb] = \gb(H_\ga)X_\gb$.\vskip 0.05in
\item[(C3)] For $\ga \in \gD$ we have $B_\fg(X_\ga, X_{-\ga})$ = 1; in particular,
$(\ga, \ga) = 2$.\vskip 0.05in
\item[(C4)] For $\ga,\gb \in \gD$ we have 
$\gb(H_\ga) = (\gb,\ga)$.\vskip 0.05in 
\end{enumerate}
\vskip 0.1in

Let $\fq$ be the parabolic subalgebra of $\fg$ of Heisenberg type; that is,
the parabolic subalgebra 
corresponding to the subset $\{\ga \in S \; | \; (\ga,\gamma) = 0 \}$. 
Denote by $\fl$ the Levi factor of $\fq$ and by $\fn$ the
nilpotent radical of $\fq$.
Then the action of $\ad(H_\gamma)$ on $\fg$ has eigenvalues 
 $-2$, $-1$, $0$, $1$, $2$, and the corresponding eigenvalue decomposition of $\fg$ is 
denoted by
\beu
\fg = \fz(\bar \fn) \oplus V^- \oplus \fl \oplus V^+ \oplus \fz(\fn).
\end{equation*}
\vskip 0.05in
\noindent Note that $V^+$ and $V^-$ are irreducible $\fl$-modules,
since the Heisenberg parabolic $\fq$ is maximal
(see \cite[Exercise 5, page 638]{Knapp02} for instance).

Let $\eD_\gamma(\fg, \fh)$ be the deleted Dynkin diagram 
associated to the Heisenberg parabolic $\fq$; 
that is, the subdiagram of the Dynkin diagram of $(\fg, \fh)$ obtained by deleting
the node corresponding to the simple root that is not orthogonal to $\gamma$,
and the edges that involve it.

As on p.789 in \cite{BKZ08} the operator $\Omega_2$ is given in terms of $R$ by  

\beu
\Omega_2(Z) =  -\frac{1}{2}\sum_{\ga, \gb \in \gD(V^+)}
N_{\gb, \gb'}M_{\ga, \gb'}(Z) R(X_{-\ga})R(X_{-\gb})
\end{equation*}
\vskip 0.1in
\noindent for $Z \in \fl$. 
It follows from Theorem 5.2 of \cite{BKZ08} and 
the data tabulated in Section 8.10 of \cite{BKZ08}
that each $\Omega_2$ system associated to a singleton component of 
$\eD_\gamma(\fg,\fh)$
is conformally invariant on the line bundle $\Cal{L}_{1}$. 
The reader may want to note here that
the special values of our $\Omega_2$ system are 
of the form $-s_0$ with $s_0$ the special values of the $\Omega_2$ system
given in \cite{BKZ08},
because the parabolic $\fq$ is chosen in this paper, 
while the opposite parabolic 
$\bar \fq$ is chosen in \cite{BKZ08}.
One can also check that we have 
$\Omega_2(\Ad(l)Z) = \chi(l) l \cdot \Omega_2(Z)$
for all $l \in L_0$. Note that this is different from the $\Ad(l)$ transformation law
that appears in \cite{BKZ08}, for the same reason.
We extend the $\C$-linear maps $d\chi$, $R$, and $\Omega_2$ to be left $C^\infty(\bar N_0)$-linear so that certain relationships can be expressed more easily.

In the rest of this paper our line bundle is assumed to be $\Cal{L}_{1}$
and for simplicity we denote $\Pi_1$ by $\Pi$.
Now we define an operator $\tilde{\Omega}_3$ on $C^\infty(\bar N_0, \C_\chi)$ by

\beu
\tilde{\Omega}_3(Y) = \sum_{\ge \in \gD(V^+)} R(X_{-\ge})\Omega_2\big( [X_\ge, Y] \big)
\end{equation*}
for $Y \in V^-$. 

\begin{Lem}\label{Lem30}
Let $W_1, \ldots, W_m$ be a basis for $V^+$ and 
$W_1^*, \ldots, W_m^*$ be the $B_\fg$-dual basis of $V^-$.
Then
\begin{equation*}
\tilde{\Omega}_3(Y) = \sum_{i=1}^m R(W_i^*)\Omega_2\big( [W_i, Y] \big).
\end{equation*}
\end{Lem}

\begin{proof}
Suppose that $\gD(V^+) = \{ \ge_1, \ldots, \ge_m\}$.
Each $W_i$ then may be expressed by 
\beu
W_i = \sum_{j = 1}^m a_{ij}X_{\ge_j}
\end{equation*}
\vskip 0.05in
\noindent for $a_{ij} \in \C$.
Let $[a_{ij}]$ be the change of basis matrix 
and set $[b_{ij}] = [a_{ij}]^{-1}$. 
Then define 
\beu
W^*_i = \sum_{k=1}^mb_{ki}X_{-\ge_k}
\end{equation*}
\vskip 0.05in
\noindent for $i=1,\ldots, m$. 
Since $B_\fg(X_{\ge_i}, X_{-\ge_j}) = \gd_{ij}$ with $\gd_{ij}$ the Kronecker delta,
it follows that 
\beu
B_\fg(W_i, W_j^*) = \gd_{ij}. 
\end{equation*}
\vskip 0.05in
\noindent Thus $\{W_1^*, \ldots, W_m^*\}$
is the dual basis of $\{W_1,\ldots, W_m\}$. 
Note that we have 
$\sum_{i = 1}^m b_{ki}a_{ij} = \gd_{kj}$
since $[b_{ij}][a_{ij}] = I$.
Then a direct computation shows that
\begin{align*}
\sum_{i=1}^m R(W_i^*)\Omega_2([W_i, Y])
&=\sum_{j,k=1}^m\big( \sum_{i=1}^m b_{ki}a_{ij}\big)R(X_{-\ge_k})\Omega_2([X_{\ge_j},Y])\\
&= \sum_{j=1}^m R(X_{-\ge_j})\Omega_2([X_{\ge_j}, Y]).
\end{align*}
\vskip 0.05in
\noindent
This completes the proof.
\end{proof}

\begin{Lem}\label{Lem31}
For all $l \in L_0$, $Z\in \fl$, and $Y \in V^-$, we have
\begin{equation}\label{Eqn32}
\tilde{\Omega}_3(\emph{\Ad}(l)Y) 
= \chi(l) l \cdot \tOmega_3(Y)
\end{equation} 
and
\beu
[\Pi(Z), \tOmega_3(Y)] = \tOmega_3\big([Z,Y]\big) - d\chi(Z)\tOmega_3(Y).
\end{equation*}
\vskip 0.1in
\end{Lem}

\begin{proof}
Recall that $l \cdot R(u) = R(\Ad(l)u)$ for $l \in L_0$ and $u \in \Cal{U}(\bar\fn)$. 
Since we have $\Omega_2(\Ad(l)W) = \chi(l)l\cdot \Omega_2(W)$ for $l \in L_0$ and $W\in \fl$,
it follows that
\begin{equation}\label{Eqn33}
\chi(l) l \cdot \tOmega_3(Y)
= \sum_{\ge \in \gD(V^+)} R(\Ad(l)X_{-\ge})\Omega_2\big( [\Ad(l)X_\ge, \Ad(l)Y] \big).
\end{equation}
By Lemma \ref{Lem30}, the value of $\tilde{\Omega}_3(Y)$ is independent 
from a choice of a basis for $V^+$. Therefore
the right hand side of (\ref{Eqn33}) is equal to the sum
$\sum_{\ge \in \gD(V^+)} R(X_{-\ge})\Omega_2\big( [X_\ge, \Ad(l)Y] \big)$,
which is $\tOmega_3(\Ad(l)Y)$. 
The second equality is obtained by differentiating the first.
\end{proof}

\begin{Prop}\label{Prop34}
We have
\beu
[\Pi(X), R(Y)]_{\nbar} = 
R\big( [\emph{Ad}(\nbar^{-1})X, Y]_{V^-} \big)_{\bar{n}} - d\chi \big( [\emph{Ad}(\nbar^{-1})X, Y]_{\fl} \big)
\end{equation*}
for all $X \in \fg$, $Y \in V^{-}$, and $\bar{n} \in \bar{N}_0$.
\end{Prop}

\begin{proof}
Let $F$ be the subspace of $\eM(\C_{-d\chi}) $ spanned by $X_{-\ga} \otimes 1$ 
and $1\otimes1$ with  $\ga \in \gD(V^+)$. A direct computation shows that 
$F$ is a $\fq$-submodule of $\eM(\C_{-d\chi}) $ and that for $Z \in \fl$ and $U \in \fn$ we have
\beu
Z(X_{-\ga} \otimes 1) = [Z, X_{-\ga}]\otimes 1 - d\chi(Z)X_{-\ga}\otimes 1
\end{equation*}
and
\beu
U(X_{-\ga} \otimes 1) = -d\chi([U, X_{-\ga}]_\fl)1\otimes 1.
\end{equation*}
Then it follows from Theorem \ref{Thm24} that if
$X \in \fg$ and $\big( \Ad(\nbar^{-1})X\big)_\fq = Z + U$ 
with $Z\in \fl$ and $U \in \fn$ then for $Y \in V^-$,
\beu
[\Pi(X), R(Y)]_{\nbar} = R\big([Z,Y]\big)_{\nbar} - d\chi\big([U,Y]\big).
\end{equation*}
Since $[Z,Y] = [\Ad(\nbar^{-1})X, Y]_{V^-}$ 
and $[U, X_{-\ga}]_\fl = [\Ad(\nbar^{-1})X, Y]_{\fl}$, this completes the proof.
\end{proof}

\vskip 0.1in

Let $\omega_2(X)$ denote the element in $\Cal{U}(\bar \fn)\otimes \C_{-d\chi}$
that corresponds to $\Omega_2(X)$ under $R$. 

\begin{Lem}\label{Lem33}
For $W, Z \in \fl$, we have
\beu
\omega_2\big( [Z, W] \big) = Z \omega_2(W) + 2d\chi(Z)\omega_2(W).
\end{equation*}
\end{Lem}

\begin{proof}
Since $\Omega_2(\Ad(l)W) = \chi(l)l\cdot \Omega_2(W)$ for $l \in L_0$,
we have $\omega_2(\Ad(l)W) = \chi(l)\Ad(l)\omega_2(W)$ by Lemma 18 in \cite{BKZ09}. 
Then the formula is obtained by replacing $l$ by $\exp(tZ)$ with $Z \in \fl_0$,
differentiating, and setting at $t = 0$.
\end{proof}

\begin{Prop}\label{Prop35}
We have 
\beu
[\Pi(X), \Omega_2(W) ] _{\nbar}
= \Omega_2\big([\emph{Ad}(\nbar^{-1})X, W]_{\fl}\big)_{\nbar} 
- d\chi \big((\emph{Ad}(\nbar^{-1})X)_\fl \big)\Omega_2(W)_{\nbar}
\end{equation*}
for all $X \in \fg$, $W \in \fl$, and $\nbar \in \bar{N}_0$.
\end{Prop}

\begin{proof}
Recall that the $\Omega_2$ system is conformally invariant on the line bundle
$\Cal{L}_{1}$. Therefore $F \equiv \text{span}_\C\{  \omega_2(W) \; | \; W \in \fl\}$ is 
a $\fq$-submodule of $\eM(\C_{-d\chi}) $. By applying Lemma \ref{Lem33} with $Z = H_\gamma$, 
we obtain $H_\gamma \omega_2(W) = -4\omega_2(W)$ for all $W \in \fl$.
For $U \in V^+$ we have
$H_\gamma U\omega_2(W) = -3U\omega_2(W)$, and 
$H_\gamma X_\gamma \omega_2(W) = -2X_\gamma \omega_2(W)$ for all $W \in \fl$.
Therefore if $U \in \fn$ then
$U\omega_2(W) = 0$ for all $W \in \fl$,
because otherwise $U\omega_2(W)$ would have the wrong $H_\gamma$-eigenvalue
to lie in $F$.
Since Lemma \ref{Lem33} shows that 
\beu
Z\omega_2(W) = \omega_2([Z,W]) - 2d\chi(Z)\omega_2(W)
\end{equation*}
for $Z,W \in \fl$, the proposed formula now follows from Theorem \ref{Thm24}.
\end{proof}

\begin{Lem}\label{Lem34}
For $X \in V^+$ and $Y \in V^-$, we have 
\beu
\sum_{\ge \in \gD(V^+)}\Omega_2\big([[X, X_{-\ge}],[X_\ge, Y]]\big)
=2\Omega_2([X,Y]).
\end{equation*}
\end{Lem}

\begin{proof}
Since we have $||\ge||^2 = 2$ for all $\ge \in \gD(V^+)$,
it follows from Proposition 2.2 of \cite{BKZ08} that
\beu
\sum_{\ge \in \gD(V^+)}\Omega_2\big( [[X, X_{-\ge}],[X_\ge, Y]] \big)
= \frac{1}{2} \sum_{\eC}p(D_4,\eC)\Omega_2 \big( \text{pr}_{\eC}([X,Y])\big),
\end{equation*}
where $\eC$ are the connected components of $\eD_\gamma(\fg, \fh)$
as in \cite{BKZ08} and $\text{pr}_{\eC}([X,Y])$ is the projection
of $[X,Y]$ onto $\fl(\eC)$, the ideal of $[\fl, \fl]$ corresponding to $\eC$.
(See Section 2 of \cite{BKZ08} for further discussion.)
One can find in Section 8.4 of \cite{BKZ08} that 
$p(D_4, \eC) = 4$ for all the components $\eC$. Then the fact that
$\Omega_2(H_\gamma) = 0$ shows that we obtain
\beu
\sum_{\ge \in \gD(V^+)}\Omega_2\big( [[X, X_{-\ge}],[X_\ge, Y]] \big)
= 2\Omega_2 \big( [X,Y]\big),
\end{equation*}
which is the proposed formula.
\end{proof}

Now with the above lemmas and propositions
we are ready to show the following key theorem.

\begin{Thm}\label{Thm36}
We have $[\Pi(X), \tilde{\Omega}_3(Y)]_e = 0$ for all $X \in V^+$ and all $Y \in V^-$.
\end{Thm}

\begin{proof}
The commutator $[\Pi(X), \tilde{\Omega}_3(Y)]$ is a sum of two terms.
One of them is given by
\begin{align}\label{Eqn37}
&\sum_{\ge\in\gD(V^+)}[\Pi(X), R(X_{-\ge})]\Omega_2\big( [X_\ge, Y]\big)\\
&=
\sum_{\ge\in\gD(V^+)}R\big([\Ad(\cdot^{-1})X, X_{-\ge}]_{V^-}\big) \Omega_2\big( [X_\ge, Y] \big)
-\sum_{\ge\in\gD(V^+)}
 d\chi\big( [\Ad(\cdot^{-1})X, X_{-\ge}]_\fl\big) \Omega_2\big([X_\ge, Y]\big),
\nonumber
\end{align}
by Proposition \ref{Prop34}. At $e$, the first term is zero,
since $[X, X_{-\ge}]_{V^-} = 0$ for all $\ge \in \gD(V^+)$.
By writing out $X$ as a linear combination of $X_\ga$
with $\ga \in \gD(V^+)$,
one can see that at the identity the second term in (\ref{Eqn37}) evaluates to
\begin{equation*}
-\sum_{\ge\in\gD(V^+)} d\chi\big( [X, X_{-\ge}]\big) \Omega_2\big([X_\ge, Y]\big)_e
=-\Omega_2\big([X,Y]\big)_e
\end{equation*}
since $d\chi(H_\ga)=1$ for $\ga \in \gD(V^+)$. 
The other term is given by
\begin{align}\label{Eqn38}
&\sum_{\ge\in\gD(V^+)} R(X_{-\ge})\big[\Pi(X), \Omega_2( [X_\ge, Y])\big]\\
&= \sum_{\ge\in\gD(V^+)} R(X_{-\ge})\Omega_2\big([\Ad(\cdot^{-1})X,[X_\ge, Y]]_\fl\big)
- \sum_{\ge\in\gD(V^+)}
 R(X_{-\ge})d\chi\big( (\Ad(\cdot^{-1})X)_\fl \big)\Omega_2\big([X_\ge,Y]\big), \nonumber
\end{align}
by Proposition \ref{Prop35}. 
To further evaluate this expression, we make use of a simple general observation.
Namely, if $D$ is a first order differential operator, $\phi$ and $\psi$ are smooth functions,
and $\phi(e) =0$ then $D_e(\phi\psi) = D_e(\phi)\psi(e)$.
Notice that $\nbar \mapsto \ad(\Ad(\nbar^{-1})X)$ is a smooth function on $\Nbar_0$.
It follows from the left $C^\infty(\bar N_0)$-linear extension of $\Omega_2$
that the first term of the right hand side of $(\ref{Eqn38})$ can be expressed as
\beu
\sum_{\ge \in \gD(V^+)}
R(X_{-\ge}) \big( \ad(\Ad(\cdot^{-1})X)_\fl \cdot \Omega_2\big([X_\ge, Y]\big) \big),
\end{equation*} 
where $\ad(\Ad(\cdot^{-1})X)_\fl$ denotes the map $Z \mapsto [\Ad(\cdot^{-1})X, Z]_\fl$
for $Z \in \fg$.
Since we have
\beu 
\big(R(X_{-\ge})\acts (\Ad(\cdot^{-1})X)\big)(e)  = [X,X_{-\ge}],
\end{equation*}
\vskip 0.05in
\noindent
$[X, [X_\ge, Y]]_\fl = 0$, and
$X_\fl = 0$,  
the right hand side of (\ref{Eqn38}) then evaluates at the identity to
\beu
\sum_{\ge\in\gD(V^+)}\Omega_2\big([[X, X_{-\ge}],[X_\ge, Y]]\big)_e
-\sum_{\ge\in\gD(V^+)} d\chi\big( [X, X_{-\ge}]\big) \Omega_2\big([X_\ge, Y]\big)_e,
\end{equation*}
which is equivalent to
\beu
\sum_{\ge\in\gD(V^+)}\Omega_2\big([[X, X_{-\ge}],[X_\ge, Y]]\big)_e
-\Omega_2\big([X,Y]\big)_e.
\end{equation*}
Therefore we obtain
\beu
[\Pi(X), \tilde{\Omega}_3(Y)]_e
= \sum_{\ge\in\gD(V^+)}\Omega_2\big([[X, X_{-\ge}],[X_\ge, Y]]\big)_e
-2\Omega_2\big([X,Y]\big)_e.
\end{equation*}
Now it follows from Lemma \ref{Lem34} that 
$[\Pi(X), \tilde{\Omega}_3(Y)]_e = 0$.
\end{proof}

\begin{Prop}\label{Prop37}
For $Y \in V^-$, we have $[\Pi(X_\gamma), \tilde{\Omega}_3(Y)]_e = 0$.
\end{Prop}

\begin{proof}
Since $\fz(\fn) = [V^+, V^+]$, it suffices to show that 
$[\Pi\big([X_1, X_2]\big), \tilde{\Omega}_3(Y)]_e = 0$ for $X_1, X_2 \in V^+$.
Note that we have $\Pi\big([X_1, X_2]\big) = [\Pi(X_1), \Pi(X_2)]$, so 
it follows from the Jacobi identity that
$[\Pi\big([X_1, X_2]\big), \tilde{\Omega}_3(Y)]$ may be 
expressed as a sum of two terms. The first is
\begin{equation*}
[\Pi(X_1), [\Pi(X_2), \tilde{\Omega}_3(Y)]] 
=\Pi(X_1)[\Pi(X_2), \tilde{\Omega}_3(Y)] - [\Pi(X_2), \tilde{\Omega}_3(Y)]\Pi(X_1).
\end{equation*}
\vskip 0.05in
\noindent By (\ref{Eqn21}), we have $\Pi(X)_e = 0$ for all $X \in \fn$. 
Using this fact and Theorem \ref{Thm36}, it is obtained that 
$[\Pi(X_1), [\Pi(X_2), \tilde{\Omega}_3(Y)]] _e = 0$ 
since $(D_1 D_2)_e = (D_1)_e  D_2$ for $D_1, D_2 \in \mathbb{D}(\Cal{L}_1)$.
The second term is 
\beu
[\Pi(X_2), [\tilde{\Omega}_3(Y), \Pi(X_1)]]
=\Pi(X_2)[\tilde{\Omega}_3(Y),\Pi(X_1)] - [\tilde{\Omega}_3(Y),\Pi(X_1)]\Pi(X_2).
\end{equation*}
\vskip 0.05in
\noindent It follows from the same argument for the first term that we have
$[\Pi(X_2), [\tilde{\Omega}_3(Y),\Pi(X_1)]]_e = 0$. 
This concludes that the proposition.
\end{proof}

\begin{Thm}\label{Thm311}
Let $\fg$ be the complex simple Lie algebra of type $D_4$, 
and $\fq$ be the parabolic subalgebra of Heisenberg type.
Then the $\tilde{\Omega}_3$ system is conformally invariant on the line bundle $\Cal{L}_1$.
\end{Thm}

\begin{proof}
For $Y \in V^-$, it follows from Lemma \ref{Lem31} that 
\beu
[\Pi(Z), \tOmega_3(Y)]_e = \tOmega_3\big([Z,Y]\big)_e - d\chi(Z)\tOmega_3(Y)_e
\end{equation*}
for all $Z \in \fl$. 
Also Theorem \ref{Thm36} and Proposition \ref{Prop37} show that
$[\Pi(U),\tOmega_3(Y)] = 0$ for all $U \in \fn$.
By the definition of $\tOmega_3(Y)$,
it is clear that $[\Pi(\bar U),\tOmega_3(Y)]_e = 0$ for all $\bar U \in \bar\fn$.
Now by applying Proposition \ref{Prop22} we conclude that the $\tOmega_3$ system
is conformally invariant on $\Cal{L}_1$. 
\end{proof}

Let $\omega_3(Y)$ denote the element in 
$\Cal{U}(\bar \fn)\otimes \C_{-d\chi}$ that corresponds to $\tOmega_3(Y)$ under $R$.
Theorem \ref{Thm311}  then implies that 
$E \equiv \text{span}_\C\{\omega_3(Y)\; | \; Y \in V^-\}$ 
is a $\fq$-submodule of $\eM(\C_{-d\chi})$.
Note that it follows from (\ref{Eqn32}) that 
we have $\omega_3(\Ad(l)Y) = \chi(l)\Ad(l)\omega_3(Y)$ for $l \in L_0$.
By using the $\Ad(l)$ transformation law,
one can check that a map $Y \otimes 1 \mapsto \omega_3(Y)$
from $V^- \otimes \C_{-d\chi}$ to $E$ is $L_0$-equivariant
with the standard action of $L_0$ on $V^- \otimes \C_{-d\chi}$.
In particular, $E$ is an irreducible $\fl$-module, because $V^-$ is $\fl$-irreducible.
Since $\omega_3$ has the same $\Ad(l)$ transformation law as $\omega_2$,
we have 
\begin{equation}\label{Eqn313}
\omega_3\big( [Z, Y] \big) = Z \omega_3(Y) + 2d\chi(Z)\omega_3(Y)
\end{equation}
for $Y \in V^-$ and $Z \in \fl$.  
The same argument in the proof of Proposition \ref{Prop35} then shows that 
$\fn$ acts on $E$ trivially.
Hence, $E$ is a leading $\fl$-type in the sense of \cite[page 808]{BKZ09}.

Now there exists a non-zero $\Cal{U}(\fg)$-homomorphism from 
a generalized Verma module $\Cal{U}(\fg) \otimes_{\Cal{U}(\fq)}E$ 
to $\eM(\C_{-d\chi})$, that is given by 
\beu
u \otimes \omega_3(Y) \mapsto u \cdot \omega_3(Y).
\end{equation*}
It follows from (\ref{Eqn313}) that $H_\gamma$ acts on $E$ by $-5$,
while it acts on $\C_{-d\chi}$ by $-2$; in particular, $E$ is not equivalent to $\C_{-d\chi}$. 
We now conclude the following corollary.

\begin{Cor}
Let $\fg$ be the complex simple Lie algebra of type $D_4$, 
and $\fq$ be the parabolic subalgebra of Heisenberg type.
Then the generalized Verma module $\eM(\C_{-d\chi})$ is reducible.
\end{Cor}


\bibliography{Draft6}

\end{document}